\documentclass[10pt,a4paper,reqno]{amsart}
\usepackage[utf8]{inputenc}
\usepackage[T1]{fontenc}       
\usepackage{ae}   
\usepackage{amsfonts}         
\usepackage{amsmath}
\usepackage{amsthm}
\usepackage{amssymb}
\usepackage{mathtools}
\usepackage{calc}
\usepackage{centernot}
\usepackage{color}
\usepackage[pdftex]{hyperref}
\usepackage{mathrsfs}
\usepackage{xargs}
\usepackage{todonotes}
\newcommandx{\huom}[2][1=]{\todo[linecolor=red,backgroundcolor=red!10,bordercolor=red,#1]{#2}}

\newcommand{\ve}{\varepsilon}
\newcommand{\N}{\mathbb{N}}
\newcommand{\C}{\mathbb{C}}
\newcommand{\R}{\mathbb{R}}

\newcommand{\Ss}{\mathbb{S}}

\newcommand{\ang}[1]{\left\langle #1 \right\rangle}

\newcommand{\p}{\partial}

\DeclareMathOperator\dv{div}

\newcommand{\pinf}{\partial_{\infty}}

\DeclareMathOperator\Ric{Ric}

\DeclareMathOperator\Hess{Hess}

\DeclareMathOperator\pcap{cap_\emph{p}}

\numberwithin{equation}{section}

\theoremstyle{plain}
\newtheorem{thm}{Theorem}[section]

\newtheorem{cor}[thm]{Corollary}
\newtheorem{prop}[thm]{Proposition}

\theoremstyle{definition}
\newtheorem{defin}[thm]{Definition}

\begin{document}

\title[Minimal graphic and $p$-harmonic functions]{Existence and non-existence of minimal 
graphic and $p$-harmonic functions}

\author{Jean-Baptiste Casteras}
\address{J.-B. Casteras, Departement de Mathematique
Universite libre de Bruxelles, CP 214, Boulevard du Triomphe, B-1050 Bruxelles, Belgium}
\email{jeanbaptiste.casteras@gmail.com}

\author{Esko Heinonen}
\address{E. Heinonen, Department of Mathematics and Statistics, P.O.B. 68 (Gustaf 
H\"allstr\"omin katu 2b), 00014 University
of Helsinki, Finland.}
\email{esko.heinonen@helsinki.fi}

\author{Ilkka Holopainen}
\address{I.Holopainen, Department of Mathematics and Statistics, P.O.B. 68 (Gustaf 
H\"allstr\"omin katu 2b), 00014 University
of Helsinki, Finland.}
\email{ilkka.holopainen@helsinki.fi}

\thanks{J.-B.C. supported by MIS F.4508.14 (FNRS)}
\thanks{E.H. supported by Jenny and Antti Wihuri Foundation.}


%58J32 Boundary value problems on manifolds 
%53C21 Methods of Riemannian geometry, including PDE methods; curvature restrictions
%31C45 Other generalizations (nonlinear potential theory, etc.) 

\begin{abstract}
We prove that every entire solution of the minimal graph equation that is bounded from below and has at most linear growth must be constant on a complete Riemannian manifold $M$ with only one end if $M$ has asymptotically non-negative sectional  curvature.  On the other hand, we prove the
existence of bounded non-constant minimal graphic and $p$-harmonic functions on rotationally symmetric
Cartan-Hadamard manifolds under optimal assumptions on the sectional curvatures.
\end{abstract}
\maketitle

\section{Introduction}
It is an interesting question to ask under which conditions on the underlying space $M$ there exist entire non-constant
bounded solutions $u\colon M\to\R$ to the minimal graph equation
  \begin{equation}\label{mingraph}
    \dv \frac{\nabla u}{\sqrt{1+|\nabla u|^2}} = 0
  \end{equation}
or to the $p$-Laplace equation
  \begin{equation}\label{plaplaceeqn}
   \dv (|\nabla u|^{p-2} \nabla u) = 0.
  \end{equation}
Namely, in $\R^n$ there is the famous Bernstein theorem which states that entire solutions of \eqref{mingraph}
are affine for dimensions $n\le 7$. Moreover, entire positive solutions in $\R^n$ are constant in all dimensions by the celebrated result due to Bombieri, De Giorgi, and Miranda \cite{bombieri}. For the $p$-harmonic equation \eqref{plaplaceeqn}
the situation is the same as for the harmonic functions, i.e. entire positive solutions in $\R^n$ are constants, the reason being the validity of 
a global Harnack's inequality.  

If the underlying space is changed from $\R^n$ to a Cartan-Hadamard manifold with sufficiently negative curvature, 
the situation changes for the both equations. The existence results have been proved by studying the
so-called \emph{asymptotic Dirichlet problem}. If $M$ is an $n$-dimensional Cartan-Hadamard manifold, it can be compactified
by adding a \emph{sphere at infinity}, $\pinf M$, and equipping the resulting space $\bar M \coloneqq M \cup 
\pinf M$ by the \emph{cone topology}. With this compactification $\bar M$ is homeomorphic to the closed
unit ball and $\pinf M$ is homeomorphic to the unit sphere $\Ss^{n-1}$. For details, see \cite{EO}. The asymptotic 
Dirichlet problem can then be stated
as follows: Given a continuous function $\theta \colon \pinf M \to \R$, find a function $u \in C(\bar M)$
that is a solution to the desired equation in $M$ and has ``boundary values'' $\theta$ on $\pinf M$.

Recently the asymptotic Dirichlet problem for minimal graph, $f$-minimal graph, $p$-harmonic and $\mathcal{A}$-harmonic equations 
has been studied for example in 
\cite{HoVa}, \cite{Va1}, \cite{Va2}, \cite{CHR}, \cite{HR_ns}, \cite{CHR1}, \cite{CHH1}, \cite{CHH2}, and \cite{He},  
where the existence of solutions 
was studied under various curvature assumptions and via different methods.
In \cite{CHR1} the existence of solutions to the minimal graph equation and to
the $\mathcal{A}$-harmonic equation was proved in dimensions $n\ge 3$ under curvature assumptions
	\begin{equation}\label{CHR_curvas}
	-\frac{\big(\log r(x))^{2\bar\ve}}{r(x)^2} \le K(P_x) \le -\frac{1+\ve}{r(x)^2 \log r(x)},
	\end{equation}
where $\ve>\bar\ve>0$, $P_x\subset T_xM$ is a $2$-dimensional subspace, $x\in M\setminus B(o,R_0)$, and $r(x) = d(o,x)$ is the distance to
a fixed point $o\in M$. In \cite{He} it was shown that
in the case of $\mathcal{A}$-harmonic functions the curvature lower bound can be replaced by a so-called
pinching condition
    \[
      |K(P_x)| \le C|K(P_x')|,
    \]
where $C$ is some constant and $P_x,P_x'\subset T_xM$.
One of our main theorems shows that in the above result the upper bound for the curvatures
is (almost) optimal, namely we prove the following.
\begin{thm}\label{bernstein} 
    Let $M$ be a complete Riemannian manifold with asymptotically non-negative sectional curvature and only one end. If 
    $u\colon M\to \R$ is a solution to the minimal graph equation \eqref{mingraph} that is bounded from below and has 
    at most linear growth, then it must be a constant.
    In particular, if $M$ is a Cartan-Hadamard manifold with asymptotically non-negative sectional curvature, the asymptotic Dirichlet problem is not solvable.
 \end{thm}
The notion of asymptotically non-negative sectional curvature (ANSC) is defined in Definition~\ref{def-notions}. 
It is worth pointing out that we do not assume, differing from previous results into this direction, the Ricci curvature to be non-negative;  see
e.g. \cite{RSS}, \cite{DJX}, \cite{DL}, \cite{Dajczer2016}.
   
Our theorem gives immediately the following corollary.
\begin{cor}\label{1st-cor}
 Let $M$ be a complete Riemannian manifold with only one end and assume that the sectional curvatures of $M$ satisfy
  \begin{equation*}
   K(P_x) \ge -\frac{C}{r(x)^2 \big(\log r(x)\big)^{1+\ve}}
  \end{equation*}
for sufficiently large $r(x)$ and for any $C>0$ and $\ve>0$.
Then any solution $u\colon M\to [a,\infty)$ with at most linear growth to the minimal graph 
equation \eqref{mingraph} must be constant.
\end{cor} 
The main tool in the proof of Theorem \ref{bernstein} is the gradient estimate in Proposition \ref{gradest},
where we obtain an upper bound for the gradient of a solution $u$ of the minimal graph equation in terms of an appropriate lower bound for 
the sectional curvature of $M$ and the growth of $u$. Under the assumptions in Theorem~\ref{bernstein}  we obtain a uniform gradient upper bound that enables us to prove a global Harnack's inequality for $u-\inf_M u$.

It is well-known that a global Harnack's inequality (for positive solutions) can be iterated to yield H\"older continuity estimates for all solutions and, furthermore, a Liouville (or Bernstein) type result for solutions with controlled growth. 
\begin{cor}\label{2nd-cor} 
Let $M$ be a complete Riemannian manifold with asymptotically non-negative sectional curvature and only one end.
 Then there exists a constant $\kappa\in (0,1]$, depending only on $n$ and on the function $\lambda$ in the (ANSC) 
condition such that every solution $u\colon M\to\R$ to the minimal graph equation
\eqref{mingraph} with
\[
\lim_{d(x,o)\to\infty}\frac{|u(x)|}{d(x,o)^{\kappa}}=0
\]
must be constant.
\end{cor}

Before turning to the existence results, we mention two closely related results by Greene and Wu \cite{GWgap}.
Firstly, in \cite[Theorem 2 and Theorem 4]{GWgap} they show that an $n$-dimensional, $n\neq2$, 
Cartan-Hadamard manifold 
with asymptotically non-negative sectional curvature is isometric to $\R^n$. Secondly, in 
\cite[Theorem 2]{GWgap}
they show that an odd dimensional Riemannian manifold with a pole $o\in M$ and everywhere non-positive or 
everywhere non-negative sectional curvature is isometric to $\R^n$ if $\liminf_{s\to\infty} s^2 k(s) = 0$, where 
$k(s) = \sup\{ |K(P_x)| \colon x\in M, \, d(o,x)=s, P_x\in T_xM \, \text{two-plane}\}$.

We point out that our results differ from those theorems of \cite{GWgap} (besides the methods) 
since we do not assume the existence 
of a pole or the manifold to be simply connected, and the (ANSC) condition allows the sectional 
curvature to change a sign. Moreover, in the following theorems we will see that, 
in order to get the result of Greene and Wu, it is necessary 
to assume $\liminf_{s\to\infty} s^2 k(s) = 0$ for all of the sectional curvatures and not only for the radial ones.

Concerning the existence results, we prove that, at least in the rotationally symmetric case, 
the curvature upper bound can be
slightly improved from \eqref{CHR_curvas}. We also point out that the proof of Theorem~\ref{thm1.5} is very 
elementary compared to the ones in \cite{CHR1} concerning the general cases.
  \begin{thm}[= Corollary \ref{rotsymmingraphcor}]\label{thm1.5}
	Let $M$ be a rotationally symmetric $n$-dimensional Cartan-Hadamard manifold whose radial sectional curvatures outside a compact set 
	satisfy the upper bounds
	 \begin{equation}\label{n=2}
	K(P_x) \le - \frac{1 + \ve }{r(x)^2 \log r(x)}, \quad \text{if } n=2 
	 \end{equation}
      	and
   	 \begin{equation}\label{n>2}
  	K(P_x) \le - \frac{1/2 + \ve }{r(x)^2 \log r(x)}, \quad \text{if } n\ge3. 	 
   	 \end{equation}
     		Then the asymptotic Dirichlet problem for the minimal graph equation \eqref{mingraph} is
	solvable with any continuous boundary data on $\pinf M$.
In particular, there are non-constant bounded entire solutions of \eqref{mingraph} in $M$.  
  \end{thm}
The rotationally symmetric $2$-dimensional case was previously considered in \cite{RTgeomdedi}, where 
the solvability of the asymptotic Dirichlet problem was proved under the curvature assumption \eqref{n=2}. 

In Section~\ref{sec4} we consider the existence of bounded non-constant $p$-harmonic functions
and prove the following.	
	\begin{thm}[= Corollary \ref{rotsymplaplacecor}]
	Let $M$ be a rotationally symmetric $n$-dimensional Cartan-Hadamard manifold, $n\ge 3$, whose radial sectional 
	curvatures satisfy the upper bound
   	 \[
     		K(P_x) \le - \frac{1/2 + \ve }{r(x)^2 \log r(x)}.
    	\]
	Then the asymptotic Dirichlet problem for the $p$-Laplace equation \eqref{plaplaceeqn},
	with $p\in (2,n)$, is solvable with any continuous boundary data on $\pinf M$.
	\end{thm}
We point out that the case $p=2$ reduces to the case of usual harmonic
functions, which were considered under the same curvature assumptions in \cite{march}.
It is also worth noting that our curvature upper bound is optimal in a sense that asymptotically
non-negative sectional curvature would imply a global Harnack's inequality for the $\mathcal{A}$-harmonic
functions and hence also for the $p$-harmonic functions, see e.g. \cite[Example 3.1]{holDuke}.
Also the upper bound of $p$ is optimal for this curvature bound, namely in Theorem \ref{parabthm} we show
that if 
  \begin{equation*}
    K_M(P_x) \ge - \frac{\alpha}{r(x)^2 \log r(x)}
  \end{equation*}
and $p=n$, the manifold $M$ is $p$-parabolic for all $0<\alpha\le1$, and if $p>n$, then $M$ is $p$-parabolic
for all $\alpha>0$. We want to point out that all entire positive $p$-harmonic functions or, more generally,
positive $\mathcal{A}$-harmonic functions (of type $p$) on $M$ must be constant if $M$ is $p$-parabolic.

\section{Preliminaries and definitions}\label{sec-preli}

We begin by giving some definitions that are needed in later sections. For the terminology 
in this section, we mainly follow \cite{grisal}, \cite{litam}, and \cite{holDuke}.

Let $(M,g)$ be a complete smooth Riemannian manifold. If $C\subset M$ is a compact set, then an unbounded component 
of $M\setminus C$ is called an \emph{end} with respect to C. We say that M has finitely many
ends if the number of ends with respect to any compact set has a uniform finite upper bound. 

If $\sigma$ is a smooth positive function on $M$, we define 
a measure $\mu$ by $d\mu = \sigma^2 d\mu_0$, where $\mu_0$ is the Riemannian measure of the metric $g$.
We will use the notation $(M,\mu)$ for the weighted manifold.
The \emph{weighted Laplace operator} $\Delta_\mu$ is a second order differential operator on $M$ 
defined as
    \begin{equation}\label{weightlaplace}
      \Delta_\mu f = \sigma^{-2} \dv (\sigma^2 \nabla f) = \dv_\mu (\nabla f),
    \end{equation}
where $\nabla$ is the gradient and $\dv$ the divergence with respect to the Riemannian metric $g$.
We call $\dv_\mu$ the \emph{weighted divergence}.

\begin{defin}\label{def-notions}
	We say that
	\begin{enumerate}
	\item[(ANSC)] $M$ has asymptotically non-negative sectional curvature if there exists a continuous 
	decreasing function $\lambda \colon [0,\infty) \to [0,\infty)$ such that 
		\[
			\int_0^\infty s\lambda(s) \,ds < \infty,
		\]
	and that $K_M(P_x) \ge -\lambda\bigl(d(o,x)\bigr)$ at any point $x \in M$; 
	
	\item[(EHI)] the weighted manifold $(M,\mu)$ satisfies the elliptic Harnack inequality if there exists a constant $C_H$ such that,
	for any ball $B(x,r)$, any positive weighted harmonic function $u$ in $B(x,2r)$ satisfies
		\[
			\sup_{B(x,r)} u \le C_H \inf_{B(x,r)} u;
		\] 
	\item[(PHI)] the weighted manifold $(M,\mu)$ satisfies the parabolic Harnack inequality if there exists a constant $C_H$ such that,
	for any ball $B(x,r)$, any positive solution $u$ to the weighted heat equation in the cylinder $Q\coloneqq
	(0,t)\times B(x,r)$ with $t=r^2$ satisfies 
		\[
			\sup_{Q^-} u \le C_H \inf_{Q^+} u,
		\]
	where 
		\[
			Q^- = (t/4,t/2) \times B(x,r/2), \quad \text{and} \quad Q^+(3t/4,t) \times B(x,r/2).
		\]
	\end{enumerate}
\end{defin}

Using the previous definitions we can now state the following  main result \cite[Theorem 1.1]{grisal} due Grigor'yan and Saloff-Coste, 
although we do not need it in its full strength. 
 
\begin{thm}\label{grisalmain}
 Let $M$ be a complete non-compact Riemannian manifold having either (a) asymptotically non-negative sectional curvature or (b) non-negative Ricci curvature outside a compact set and finite first Betti number. Then $M$ satisfies (PHI) if and only if it satisfies (EHI). Moreover, (PHI) and (EHI) hold if and only if either $M$ has only one end or $M$ has more that one end and the functions $V$ and $V_i$ satisfy for large enough $r$ the conditions
 $V_i(r)\approx V(r)$ (for all indices $i$) and 
 \[
\int_1^r\frac{s\,ds}{V(s)}\approx \frac{r^2}{V(r)}. 
 \]
\end{thm}
Above $V(r)=\mu\bigl(B(o,r)\bigr)$ and $V_i(r)=\mu\bigr(B(o,r)\cap E_i\bigr)$ for an end $E_i$.

We will briefly sketch the rather well-known proof of the validity of (EHI) in the case where $M$ has only one end. For that purpose, we need additional 
definitions.
\begin{defin}
 We say that
 \begin{enumerate}
	\item[(VD)] a family $\mathcal{F}$ of balls in $(M,\mu)$ satisfies the volume doubling property if there exists a 
	constant $C_D$ such that for any ball $B(x,r)\in\mathcal{F}$ we have
	  \[
	    \mu \big(B(x,r) \big) \le C_D \mu\big( B(x,r/2) \big);
	  \]	
	\item[(PI)] a family $\mathcal{F}$ of balls in $(M,\mu)$ satisfies the Poincar\'e inequality if there exists a constant $C_P$ such that
	for any ball $B(x,r)\in\mathcal{F}$ and any $f \in C^1(B(x,r))$ we have
	  \[
	    \inf_{\xi\in\R} \int_{B(x,r)} (f-\xi)^2 \, d\mu \le C_P r^2 \int_{B(x,r)} |\nabla f|^2 \, d\mu;
	  \]
	\item[(BC)] a set $A\subset\partial B(o,t)$ has a ball-covering property if, for each $0<\varepsilon<1$, $A$ can be covered by 
	  $k$ balls of radius $\varepsilon t$ with centres in $A$, where $k$ depends on $\varepsilon$ and possibly on some other parameters, 
	  but is independent of $t$.
\end{enumerate}
\end{defin}

From the curvature assumptions ((a) or (b)) in Theorem~\ref{grisalmain} it follows that (VD) and (PI) hold for all ``remote''
balls, that is for balls $B(x,r)$, where $r\le\frac{\varepsilon}{2}d(o,x)$ and $\varepsilon\in (0,1]$ is a suitable remote parameter. The familiar Moser iteration then yields local (EHI) for such remote balls. Furthermore, if $E$ is an end of $M$ and $E(t)$ denotes the unbounded component of $E\setminus\bar{B}(o,t)$, then set $\partial E(t)$ is connected and has the ball-covering property (BC) for all sufficiently large $t$. Iterating the local (EHI)
$k$ times, one obtains Harnack's inequality 
\[
\sup_{\partial E(t)}u\le C\inf_{\partial E(t)}u
\] 
with $C$ independent of $t$. Finally, if $M$ has only one end, the global (EHI) follows from the maximum principle. We will give a bit more details in 
Section~\ref{sec-nonexist} and refer to 
\cite{grisal}, \cite{litam}, and \cite{holDuke} for more details, and to \cite{Ab}, \cite{Cai}, \cite{kas}, \cite{litam}, and \cite{liu}
for the connectivity and the covering properties mentioned above.

\section{Non-existence for minimal graph equation}\label{sec-nonexist}

In order to prove Theorem \ref{bernstein} we need a uniform gradient estimate for the 
solutions of the minimal graph equation \eqref{mingraph}. Our proof follows closely the 
computations in \cite{DL} and \cite{RSS}. 
It is worth pointing out that the solutions in Theorem~\ref{bernstein} need not be bounded and that we do not assume the Ricci curvature being non-negative, therefore the gradient estimates in \cite[Theorem 1.1]{Spruck} and \cite[Theorem 4.1]{RSS} are not available as such in our setting.
We begin by introducing some notation.

We assume that $M$ is a complete non-compact $n$-dimensional Riemannian manifold whose Riemannian metric is given by
$ds^2=\sigma_{ij}dx^i dx^j$ in local coordinates.
Let $u\colon M \to \R$ be a solution to the minimal graph equation, i.e.
    \[
      \dv \frac{\nabla u}{\sqrt{1+|\nabla u|^2}} = 0,
    \]
where the gradient and divergence are taken with respect to the Riemannian metric of $M$. We denote by
    \[
      S=\big\{(x,u(x)) \colon x \in M\big\}
    \]
the graph of $u$ in the product manifold $M\times\R$ and by 
\begin{equation*}\label{unitnorm}
      N= \frac{-u^j\p_j + \p_t}{W}
    \end{equation*}
the upward pointing unit normal to the graph of $u$ expressed in terms of a local coordinate frame $\{\partial_1,\ldots,\partial_n\}$ and 
$\partial_t=e_{n+1}$. Here $W = \sqrt{1+|\nabla u|^2}$ and $u^i = \sigma^{ij} D_j u$, $D_j$ being the covariant derivative on $M$.
The components of the induced metric on the graph are given by $g_{ij}=\sigma_{ij} + u_i u_j$ with inverse     
\begin{equation*}
      g^{ij} = \sigma^{ij} - \frac{u^i u^j}{W^2}.
    \end{equation*}
We denote by $\nabla^S$ and $\Delta^S$ the gradient and, respectively, the Laplace-Beltrami operator on the graph $S$. 
For the Laplacian on the graph we have the Bochner-type
formula (see e.g. \cite{fornari})
     \begin{equation}\label{bochner}
      \Delta^S\ang{e_{n+1},N} = -\big( |A|^2 + \overline{\Ric}(N,N) \big)\ang{e_{n+1},N},
     \end{equation}
where $|A|$ is the norm of the second fundamental form and $\overline{\Ric}$ is the Ricci curvature of $M\times\R$. 
From \eqref{bochner} we obtain
      \begin{equation}\label{bochW}
       \Delta^S W = 2 \frac{|\nabla^S W|^2}{W} + \big( |A|^2 + \overline{\Ric}(N,N) \big) W.
      \end{equation}
Here and in what follows we extend, without further notice, functions $h$ defined on $M$ to $M\times\R$ by setting $h(x,t)=h(x)$.
The Laplace-Beltrami operator of the graph can be expressed in local coordinates as
    \[
      \Delta^S = g^{ij}D_i D_j.
    \]
   
Now we are ready to prove the following gradient estimate.
  \begin{prop}\label{gradest}
Assume that the sectional curvature of $M$ 
has a lower bound $K(P_x)\ge -K_0^2$ for all $x\in B(p,R) $ for some constant $K_0=K_0(p,R)\ge 0$.
Let $u$ be a positive solution to the minimal graph equation in $B(p,R) \subset M$. Then 
      \begin{align}\label{loc-grad-est}
       & |\nabla u(p)| \\
       &\quad \le \left(\frac{2}{\sqrt{3}}+\frac{32 u(p)}{R}\right)
       \left( \exp\left[64 u(p)^2 \left( \frac{2\psi(R)}{R^2}+\sqrt{\frac{4\psi(R)^2}{R^4} 
       + \frac{(n-1)K_0^2}{64 u(p)^2}} \right)\right] + 1\right),\nonumber
             \end{align}
    where $\psi(R) = (n-1)K_0 R \coth(K_0R) + 1$ if $K_0>0$ and $\psi(R)=n$ if $K_0=0$. 
  \end{prop}

    \begin{proof}
    Define a function $h = \eta W$, where $\eta(x) = g(\varphi(x))$ with $g(t) = e^{Kt}-1$,  
    	\[
		\varphi(x) = \left( 1- \frac{u(x)}{4u(p)} - \frac{d(x,p)^2}{R^2} \right)^+,
	\]
	and a constant $K$ that will be specified in \eqref{Kchoose}.
    Denote by $C(p)$ the cut-locus of $p$ and let $U(p) = B(p,R) \setminus C(p)$. Then it is well known
    that $d(x,p)$ is smooth in the open set $U(p)$. We assume that the function $h$ attains its maximum 
    at a point $q\in U(p)$, and for the case
    $q\not\in U(p)$ we refer to \cite{RSS}.
    
    In all the following, the computations will be done at the maximum point $q$ of $h$. We have
	\begin{equation}\label{grad-h}
	  \nabla^S h = \eta \nabla^S W + W\nabla^S \eta = 0
	\end{equation}
    and since the Hessian of $h$ is non-positive, we obtain, using \eqref{grad-h} and \eqref{bochW},
      \begin{align}\label{Laplacehestim}
       0\ge \Delta^S h &= W \Delta^S \eta + 2\ang{\nabla^S\eta,\nabla^S W} + \eta \Delta^S W \nonumber \\
	&= W \Delta^S \eta + \left( \Delta^S W - \frac{2}{W} |\nabla^S W|^2 \right) \eta  \\
	&= W \big( \Delta^S \eta + (|A|^2 + \overline{\Ric}(N,N))\eta \big),\nonumber
      \end{align}
    where $\overline{\Ric}$ is the Ricci curvature of $M\times\R$. Since the Ricci curvature of $M\times\R $ in 
$B (p,R)\times\R $ has a lower bound $\overline {\Ric}(N,N)\ge -(n-1)K_0^2$, we obtain from \eqref{Laplacehestim} that 
    $\Delta^S \eta \le (n-1)K_0^2\eta$ 
    and hence, from the definition of $\eta$, we get
	\begin{equation}\label{laplacephi}
	 \Delta^S \varphi + K|\nabla^S \varphi|^2 \le \frac{(n-1)K_0^2}{K}.
	\end{equation}
Next we want to estimate $\Delta^S \varphi$ from below by using the lower bound for the sectional curvature
and the Hessian comparison theorem. For this, let $\{e_i\}$ be a local orthonormal
frame on $S$. Since $u$ is a solution to the minimal graph equation, we have $\Delta^S u =0$ and
\[
\sum_{i=1}^n \ang{\bar\nabla_{e_i}N,e_i} =0,
\]
where $\bar{\nabla}$ denotes the Riemannian connection of the ambient space $M\times\R$.
Hence
\begin{align*}
	\Delta^S \varphi &= \Delta^S\left( -\frac{d^2}{R^2} \right) =- \frac{1}{R^2} \sum_{i=1}^n 
	\ang{\nabla^S_{e_i}\nabla^S d^2,e_i} \\
	&= - \frac{1}{R^2} \sum_{i=1}^n \ang{\bar\nabla_{e_i}\bigl(\bar\nabla d^2-\ang{\bar\nabla d^2,N}N\bigr),e_i} \\
	&= - \frac{1}{R^2} \sum_{i=1}^n \ang{\bar\nabla_{e_i}\bar\nabla d^2,e_i} \\
	&= - \frac{2d}{R^2} \sum_{i=1}^n \ang{\bar\nabla_{e_i}\bar\nabla d,e_i}
	- \frac{2}{R^2} \sum_{i=1}^n (e_i d)\ang{\bar\nabla d,e_i} \\
	&\ge - \frac{2d}{R^2} \sum_{i=1}^n \ang{\bar\nabla_{e_i}\bar\nabla d,e_i} -\frac{2}{R^2}.
	\end{align*}
Now decompose $e_i$ as $e_i = (e_i - \ang{\p_t,e_i}\p_t) + \ang{\p_t,e_i}\p_t \eqqcolon \hat{e}_i + \ang{\p_t,e_i}\p_t.$ 	
Then
	\begin{align*}
	\ang{\bar\nabla_{e_i}\bar\nabla d,e_i} &= \ang{\bar\nabla_{\hat e_i 
		+ \ang{\p_t,e_i}\p_t}\bar\nabla d,\hat e_i + \ang{\p_t,e_i}\p_t} \\
	&= \ang{\bar\nabla_{\hat e_i}\bar\nabla d,\hat e_i} 
	= \Hess d(\hat e_i, \hat e_i)
	\end{align*}
and by the Hessian comparison (e.g. \cite[Theorem A]{GW}) we have
\[
\Hess d (\hat{e}_i,\hat {e}_i)\le \frac {f^\prime(d)}{f (d)}\bigl( |\hat {e}_i|^2-\langle\nabla d,\hat{e}_i\rangle\bigr),
\]
where $f(t)=K_0^{-1}\sinh (K_0 t) $ if $K_0>0$ and $f (t)=t $ if $K_0=0$. Choosing $\hat{e}_n $ parallel to 
$\nabla d $ at $q $ we have
\[
\Hess d (\hat e_i,\hat e_i)\le
\begin{cases}
0, &\text{ if }i=n;\\
\frac{f^\prime(d)}{f(d)}, &\text{ if }i\in\{1,\ldots,n-1\}.
\end{cases}
\]
Hence
\begin{align*}
\sum_{i=1}^n \ang{\bar\nabla_{e_i}\bar\nabla d,e_i} & =\sum_{i=1}^n \Hess d(\hat e_i, \hat e_i)\\
&
\le\begin{cases}
(n-1)K_0\coth(K_0 d), &\text{ if } K_0>0;\\
\frac{n-1}{d}, &\text{ if } K_0=0.
\end{cases}
\end{align*}
Therefore
\begin{equation}\label{comparison}
\Delta^S \varphi\ge - \frac{2d}{R^2} \sum_{i=1}^n \ang{\bar\nabla_{e_i}\bar\nabla d,e_i}-\frac{2}{R^2}
 \ge -\frac{2 \psi(R)}{R^2},
\end{equation}	
where $\psi$ is as in the claim.

A straightforward computation gives also 
 	\begin{align*}
	|\nabla^S \varphi|^2 = g^{ij} D_i \varphi D_j \varphi &= 
	 \frac{|\nabla u|^2}{16 u(p)^2 W^2} + \frac{4d(x,p)^2}{R^4}\left(1- 
	 	\ang{\frac{\nabla u}{W},\nabla d(x,p)}^2 \right) \\
	 &\qquad +\frac{d(x,p)}{u(p)R^2W^2}\ang{\nabla u,\nabla d(x,p)} \\	
	&\ge \frac{|\nabla u|^2}{16 u(p)^2 W^2} + \frac{4d(x,p)^2}{R^4} \left( 1-\frac{|\nabla u|^2}{W^2} \right)
		-\frac{d(x,p)|\nabla u|}{u(p)R^2W^2} \\
	&= \left( \frac{|\nabla u|}{4u(p) W} - \frac{2d(x,p)}{R^2 W} \right)^2.
	\end{align*}   
Note that 
	\begin{equation}\label{estimToAlpha}
	\left( \frac{|\nabla u|}{4u(p) W} - \frac{2d(x,p)}{R^2 W} \right)^2 > \frac{1}{16 u(p)^2 \alpha^2}
	\end{equation}
with some constant $\alpha>2$ if and only if
	\[
	\left( \frac{|\nabla u|}{4u(p)W} -\frac{2d(x,p)}{R^2 W} -\frac{1}{4u(p)\alpha} \right) 
	\left( \frac{|\nabla u|}{4u(p)W} -\frac{2d(x,p)}{R^2 W} + \frac{1}{4u(p)\alpha} \right) > 0.
	\]
This is clearly true if the first factor is positive, i.e. if
	\[
	\alpha |\nabla u| - W > \frac{\alpha 8 d(x,p) u(p)}{R^2}.
	\]
On the other hand,
	\[
	\alpha |\nabla u| - W > W
	\]
if 
	\[
	W^2 > \frac{\alpha^2}{\alpha^2-4}.
	\]
Therefore assuming
	\begin{equation*}
	W(q) > \max\left\{ \frac{\alpha}{\sqrt{\alpha^2 -4}}, \frac{\alpha 8 u(p)}{R} \right\}
	\end{equation*}
we see that also \eqref{estimToAlpha} holds and thus 
 we have the estimate 
 \begin{equation}\label{gradphiest}
  |\nabla^S \varphi|^2 > \frac{1}{16u(p)^2\alpha^2}.
 \end{equation}   
    
 Plugging \eqref{comparison} and \eqref{gradphiest} into \eqref{laplacephi} we obtain
 	\[
	-\frac{2\psi(R)}{R^2} + \frac{K}{16 u(p)^2 \alpha^2} < \frac{(n-1)K_0^2}{K}.
	\]
  But choosing 
  	\begin{equation}\label{Kchoose}
 K =  8u(p)^2\alpha^2 \left( \frac{2\psi(R)}{R^2} + \sqrt{\frac{4\psi(R)^2}{R^4} 
	+\frac{(n-1)K_0^2}{4u(p)^2\alpha^2}} \right) 	
  	\end{equation}
	with $\alpha = 4$
  we get a contradiction and hence we must have
	\[
 	W(q) \le \max\left\{ \frac{2}{\sqrt{3}}, \frac{32 u(p)}{R} \right\}.
 	\]
  This implies 
  	\begin{align*}
	h(p) = (e^{K\varphi(p)} -1) W(p) &= (e^{\frac{3}{4}K} -1)W(p) \le h(q) \\
		&\le (e^K - 1) 
	\max\left\{ \frac{2}{\sqrt{3}}, \frac{32 u(p)}{R} \right\} \\
	&\le (e^K-1)\left(\frac{2}{\sqrt{3}}+\frac{32 u(p)}{R}\right)
	\end{align*}
  and noting that $e^{\frac{3}{4}K} -1 \ge e^{\frac{K}{2}} -1$
  we obtain the desired estimate 
  \[
  |\nabla u(p)| \le (e^{\frac{K}{2}}+1)\left(\frac{2}{\sqrt{3}}+\frac{32 u(p)}{R}\right).
  \]
      \end{proof}
  Next we apply Proposition~\ref{gradest} to the setting of Theorem~\ref{bernstein} to obtain a uniform gradient estimate.
  \begin{cor}\label{cor-glob-grad-est}
  Let $M$ be a complete Riemannian manifold with asymptotically non-negative sectional curvature. 
  If $u\colon M\to \R$ is a solution to the minimal graph equation \eqref{mingraph} that is bounded from below and has 
    at most linear growth, then there exist positive constants $C$ and $R_0$ such that
    \begin{equation}\label{glob-grad-est}
    |\nabla u(x)|\le C
    \end{equation}
    for all $x\in M\setminus B(o,R_0)$.
  \end{cor}
  \begin{proof}
We may assume, without loss of generality, that $u>0$. Then the assumptions on the growth of $u$ and the (ANSC) condition  of $M$ imply that
  there exist constants $c$ and $R_0$ such that
\begin{equation}\label{ubound}
u(x)\le c\,d(x,o)
\end{equation}
and
\[
K(P_x)\ge -\frac{c}{d(x,o)^2}
\]
for all $x\in M\setminus B(o,R_0/2)$.  
Next we apply 
Proposition~\ref{gradest} to points $p\in M\setminus B(o,R_0)$ with the radius $R=d(p,o)/2\ge R_0/2$. Noticing that 
$B(p,R)\subset M\setminus B(o,R)\subset M\setminus B(o,R_0/2)$, we obtain an upper bound
\begin{equation}\label{Kbound}
K_0^2=K_0(p,R)^2\le c^2/R^2
\end{equation} 
for the constant $K_0$ in the sectional curvature bound in $B(p,R)$. It follows now from \eqref{ubound} and \eqref{Kbound} 
that 
\begin{align*}
\frac{u(p)}{R}&\le 2c,\\
\psi(R)&\le (n-1)c\,\coth(c)+1,\\
\noalign{and}
u(p)^2K_0^2&\le 4c^3.
\end{align*}
Plugging these upper bounds into \eqref{loc-grad-est} gives the estimate \eqref{glob-grad-est}. 
  \end{proof}

   We are now ready to prove the Theorem \ref{bernstein}.
  
  \begin{proof}[Proof of Theorem \ref{bernstein}]
    Denoting 
    \[
      A(x) = \frac{1}{\sqrt{1+|\nabla u|^2}}
    \]
  we see that 
    \[
      \dv \frac{\nabla u}{\sqrt{1+|\nabla u|^2}} = \dv \big( A(x)\nabla u \big) =0
    \]
  is equivalent to 
    \[
      \frac{1}{A(x)} \dv \big( A(x)\nabla u \big) =0.
    \]
  Now we can interpret the minimal graph equation as a weighted Laplace equation $\Delta_\sigma$ with the weight 
      \[
	\sigma = \sqrt{A}.
      \]
  Note that due to the uniform gradient estimate \eqref{glob-grad-est} of Corollary~\ref{cor-glob-grad-est} there exists a 
  constant $c>0$ such that $c\le  \sigma \le 1$ in $M\setminus B(o,R_0)$ and hence the operator $\Delta_\sigma$ is 
  uniformly elliptic there. 
On the other hand, the assumption (ANSC) implies that the (unweighted) volume doubling condition (VD) and the (unweighted) Poincar\'e inequality
(PI) hold  for balls inside $B(p,R)$, with $R=d(o,p)/2\ge R_0$.
More precisely, the Ricci curvature of $M$ satisfies 
\begin{equation}\label{ric-decay}
\Ric(x)\ge -\frac{(n-1)K^2}{d(x,o)^2}
\end{equation}
for some constant $K\ge 0$ if $d(x,o)\ge R_0$ and $R_0$ is large enough. Then for each $x\in B(p,R)$, with $d(o,p)\ge 2R\ge 2R_0$, we have
$\Ric(x)\ge -(n-1)\tilde{K}^2$, where $\tilde{K}=KR^{-1}$. Then the well-known Bishop-Gromov comparison theorem (see \cite[5.3.bis Lemma]{Gro})
implies that
\begin{equation}\label{BGdoubl}
\frac{\mu_0\bigl(B(x,2r)\bigr)}{\mu_0\bigl(B(x,r)\bigr)}\le 2^n\exp\bigl(2r(n-1)\tilde{K}\bigr)\le 2^n\exp\bigl((n-1)K\bigr)
\end{equation} 
for all balls $B(x,2r)\subset B(p,R)\subset M\setminus B(o,R_0)$. On the other hand, it follows from Buser’s isoperimetric inequality 
\cite{buser} that
\begin{equation}\label{B-PI}
\int_{B(x,r)}|f-f_{B(x,r)}|\,d\mu_0\le r\exp\bigl(c_n(1+\tilde{K}r)\bigr)\int_B |\nabla f|\,d\mu_0\le 
cr\int_B|\nabla f|\,d\mu_0,
\end{equation}
for every $f\in C^1\bigl(B(x,r)\bigr)$, 
where 
\[
f_{B(x,r)}=\frac{1}{\mu_0\bigl(B(x,r)\bigr)}\int_{B(x,r)}f\,d\mu_0
\]
and the constant $c$ also has an upper bound that depends only on $n$ and $K$.
 Since $\Delta_\sigma$ is uniformly elliptic in $M\setminus B(o,R_0)$, the Moser iteration method gives a local Harnack's inequality
\begin{equation}\label{lochar}
\sup_{B(p,R/2)}u\le c\,\inf_{B(p,R/2)}u
\end{equation}
for all $p\in\partial B(o,2R)$, with the constant $c$ independent of $p$ and $R$. Since we assume that $M$ has only one end, 
the boundary of the unbounded component of $M\setminus \bar{B}(o,2R)$ is connected for all sufficiently large $R$ and can be covered by $k$ balls $B(x,R/2)$, with $x\in\partial B(o,2R)$ and $k$ independent of $R$; see \cite{Ab} and \cite{kas}.
Iterating the Harnack inequality \eqref{lochar} $k$ times and applying the maximum principle  we obtain
\begin{equation}\label{globhar}
\sup_{B(o,2R)}u\le C\inf_{B(o,2R)}u.
\end{equation}
Finally, we may suppose, without loss of generality, that $\inf_M u=0$. Letting then $R\to\infty$, we get
  \[
\sup_{B(o,2R)}u\le C\inf_{B(o,2R)}u\to 0
\]
  as $R\to\infty$, and therefore $u$ must be constant.
\end{proof}
 
\begin{proof}[Proof of Corollary \ref{2nd-cor}] 
In the proof below, the constants $c,\ C, C_0, \Lambda$, and $\kappa$ depend only on $n$ and on the function
$\lambda$ in the (ANSC) assumption.

We may assume that $u(o)=0$. 
Suppose first that $u\colon M\to\R$ is a solution to the minimal graph equation \eqref{mingraph} such that
\begin{equation}\label{alkuehto}
\lim_{d(x,o)\to\infty}\frac{|u(x)|}{d(x,o)}=0.
\end{equation}
Then there exists a sufficiently large $R_0$ such that 
$|u(x)|\le d(x,o)$ for all $x\in M\setminus B(o,R_0/2)$ and that \eqref{Kbound} holds, i.e. $K_0^2=K_0(p,R)^2\le c^2/R^2$
for all $p\in M\setminus B(o,R_0)$ and $R=d(p,o)/2\ge R_0/2$.
Denote
\[
M(t)=\sup_{B(o,t)}u\quad\text{and}\quad
m(t)=\inf_{B(o,t)}u
\]
for $t>0$. Then $u-m(2t)$ is a positive solution in $B(o,2t)$ and, moreover, $u(x)-m(2t)\le 4t$ for all $x\in \partial B(o,3t/2)$
and $t\ge R_0$. Applying  
Corollary~\ref{cor-glob-grad-est} 
to $u-m(2t)$ in balls $B(x,t/2)$, where $x\in\partial B(o,3t/2)$ and $t\ge R_0$,
we obtain a uniform gradient bound
\[
|\nabla u(x)|\le C
\]
for all $x\in M\setminus B(o,3R_0/2)$. Therefore, we may apply the Harnack inequality \eqref{globhar}
to functions $u-m(2t)$, for all sufficiently large $t$, to obtain
\begin{equation}\label{har-appl}
M(t)-m(2t)\le C_0\bigl(m(t)-m(2t)\bigr).
\end{equation}
Then we proceed as in the proof of the H\"older continuity estimate for $\mathcal{A}$-harmonic functions in \cite[6.6. Theorem]{HKM} to obtain 
\begin{equation}\label{(6.9)}
M(t)-m(t)\le\Lambda\bigl(M(2t)-m(2t)\bigr),
\end{equation}
where $\Lambda=(C_0-1)/C_0$. For reader's convenience we give the short proof of  
\eqref{(6.9)}. To obtain \eqref{(6.9)} suppose first that 
\begin{equation}\label{c0inv}
m(t)-m(2t)\le C_0^{-1}\bigl(M(2t)-m(2t)\bigr).
\end{equation}
Then 
\begin{align*}
M(t)-m(t) & = M(t)-m(2t)+m(2t)-m(t)\\
&\le (C_0-1)\bigl(m(t)-m(2t)\bigr)\\
&\le \Lambda\bigl(M(2t)-m(2t)\bigr)
\end{align*}
by \eqref{har-appl} and \eqref{c0inv}.
On the other hand, if 
\[
m(t)-m(2t)\ge C_0^{-1}\bigl(M(2t)-m(2t)\bigr),
\]
then 
\begin{align*}
M(t)-m(t)&\le M(2t)-m(t)-\bigl(m(t)-m(2t)\bigr)\\
&\le \Lambda\bigl(M(2t)-m(2t)\bigr).
\end{align*}
Thus \eqref{(6.9)} always holds. Suppose then that $R\ge r$, with $r$ sufficiently large. Choose the integer $m\ge 1$ such that
$2^{m-1}\le R/r\le 2^m$. Then
\begin{align*}
M(r)-m(r)&\le \Lambda^{m-1}\bigl(M(2^{m-1}r)-m(2^{m-1}r)\bigr)\\
&\le \Lambda^{m-1}\bigl(M(R)-m(R)\bigr).
\end{align*}
Setting $\kappa=(-\log\Lambda)/\log 2$, we get 
$(r/R)^\kappa\ge 2^{-\kappa}\Lambda^{m-1}$, and therefore
\begin{equation}\label{(6.7)}
M(r)-m(r)\le 2^{\kappa}\left(\frac{r}{R}\right)^{\kappa}\bigl(M(R)-m(R)\bigr)
\end{equation}
for every $R\ge r$, with $r$ sufficiently large. 
Notice that \eqref{(6.7)} holds for all entire solutions satisfying \eqref{alkuehto}. Finally, if $u$ is an entire solution to \eqref{mingraph} 
such that
\[
\lim_{d(x,o)\to\infty}\frac{|u(x)|}{d(x,o)^{\kappa}}=0,
\]
the estimate 
\eqref{(6.7)} holds for $u$. Letting $R\to\infty$ in \eqref{(6.7)}, we obtain $M(r)-m(r)=0$ for all $r$ and, consequently, $u$ must be constant.
\end{proof}
   
\section{Existence results on rotationally symmetric manifolds}\label{sec4}

In this section we assume that $M$ is a rotationally symmetric Cartan-Hadamard manifold with the 
Riemannian metric given by
    \[
     ds^2 = dr^2 + f(r)^2 d\vartheta^2
    \]
where $r(x)=d(o,x)$ is the distance to a fixed point $o\in M$ and $f\colon(0,\infty)\to(0,\infty)$
is a smooth function with $f''\ge 0$. 
Then the (radial) sectional curvature of 
$M$ is given by $K(r) = -f''(r)/f(r)$.

On such manifold the Laplace operator can be written as
    \begin{equation}\label{laplacepolar}
     \Delta = \frac{\p^2}{\p r^2} + (n-1) \frac{f'\circ r}{f\circ r} \frac{\p}{\p r} + 
     \frac{1}{(f\circ r)^2} \Delta^{\Ss},
    \end{equation}
where $\Delta^{\Ss}$ is the Laplacian on the unit sphere $\Ss^{n-1}\subset T_oM$. For the gradient of a function $\varphi$ 
we have
    \begin{equation}\label{gradientpolar}
     \nabla \varphi = \frac{\p \varphi}{\p r} \frac{\p}{\p r} + \frac{1}{f(r)^2}\nabla^{\Ss}\varphi 
    \end{equation}
and
    \[
     |\nabla \varphi|^2 = \varphi_r^2 + f^{-2} |\nabla^{\Ss}\varphi|^2.
    \]
Here $\nabla^{\Ss}$ is the gradient on $\Ss^{n-1}$, $|\nabla^{\Ss}\varphi|$ denotes the norm of $\nabla^{\Ss}\varphi$ with respect to the Euclidean metric on $\Ss^{n-1}$,
 and $\varphi_r=\partial\varphi/\partial r$. More precisely, in geodesic polar 
coordinates $(r,\vartheta)$,
\begin{align*}
 \Delta\varphi(r,\vartheta) &= \frac{\p^2\varphi(r,\vartheta)}{\p r^2} + (n-1) \frac{f'(r)}{f(r)} \frac{\p\varphi(r,\vartheta)}{\p r} + 
     \frac{1}{f(r)^2} \Delta^{\Ss}\tilde{\varphi}(\vartheta),\\
\nabla\varphi(r,\vartheta)&=\frac{\partial\varphi(r,\vartheta)}{\partial r}\frac{\partial}{\partial r}
+\frac{1}{f(r)^2}\nabla^{\Ss}\tilde{\varphi}(\vartheta)\in \R\oplus T_{\vartheta}\Ss^{n-1},
\end{align*}
where $\tilde{\varphi}\colon\Ss^{n-1}\to\R,\ \tilde{\varphi}(\vartheta)=\varphi(r,\vartheta)$ for each fixed $r>0$.

Existence of non-constant bounded harmonic functions on rotationally symmetric manifolds was considered 
in \cite{march}, where March proved, with probabilistic arguments, that such functions exist if and only if  
    \[
     J(f)\coloneqq \int_1^\infty \Big( f^{n-3}(r)\int_r^\infty f^{1-n}(\rho) d\rho\Big)dr < \infty.
    \]
In terms of radial sectional curvature we have (for the proof see \cite{march})
    \[
J(f) < \infty \quad \text{if } K(r) \le -\frac{c}{r^2\log r} \text{ for } c>c_n \text{ and large } r,
    \]
and
    \[
     J(f) = \infty \quad \text{if } K(r) \ge -\frac{c}{r^2\log r} \text{ for } c<c_n \text{ and large } r,
    \]
where $K(r) = -f''(r)/f(r)$ and $c_2=1, \, c_n=1/2$ for $n\ge 3$.
Another proof for the existence was given in \cite{Va_lic} and our approach in this section is similar to that 
one.

\subsection{Minimal graph equation}  

First we consider the minimal graph equation and prove the following existence result.

\begin{thm} \label{rotsymmingraph}
Assume that 
	\begin{equation}\label{intcond}
	\int_1^{\infty}\Big( f(s)^{n-3} \int_s^{\infty} f(t)^{1-n} dt \Big) ds < \infty.
	\end{equation}
	Then there exist non-constant bounded solutions of the minimal graph equation and, moreover,
	the asymptotic Dirichlet problem for the minimal graph equation is uniquely solvable for any 
	continuous boundary data on $\pinf M$.
\end{thm} 

\begin{proof}
First, changing the order of integration, the condition \eqref{intcond} reads
        \begin{equation}\label{intcond2}
                \int_1^\infty \frac{\int_1^t f(s)^{n-3} ds}{f(t)^{n-1}} dt <\infty.
        \end{equation}
Now we interpret $\pinf M$ as $\Ss^{n-1}$ and let $b\colon \Ss^{n-1} \to \R$ be a smooth non-constant
function and define $B\colon M \setminus\{o\} \to \R$,
\[
 B(\exp(r\vartheta))=B(r,\vartheta) = b(\vartheta), \, \vartheta \in \Ss^{n-1}\subset T_oM.
 \] 
 Define also
        \[
                \eta(r) = k \int_r^\infty f(t)^{-n+1} \int_1^t f(s)^{n-3} ds dt,
        \]
with $k>0$ to be determined later, and note that by the assumption \eqref{intcond2} $\eta(r) \to 0$ as $r\to\infty$.

The idea in the proof is to use the functions $\eta$ and $B$, and condition \eqref{intcond2} to construct
barrier functions for the minimal graph equation to show the existence of solutions that extends continuously to the 
asymptotic boundary $\pinf M$ with prescribed asymptotic behaviour.

Begin by noticing that
  \begin{equation*}
   \eta'(r) = - k f(r)^{-n+1} \int_1^r f(s)^{n-3} ds < 0,
  \end{equation*}
  \begin{equation*}
   \eta''(r) = k(n-1) f'(r)f(r)^{-n} \int_1^r f(s)^{n-3}ds - k f^{-2}(r),
  \end{equation*}
and
 \begin{equation*}
   \Delta\eta = - k f^{-2}
  \end{equation*}
  where $\eta(x)\coloneqq\eta\bigl(r(x)\bigr)$.
The minimal graph equation for $\eta + B$ can be written as
    \begin{align}\label{mineqetaB}\begin{split}
      \dv \frac{\nabla(\eta+B)}{\sqrt{1+|\nabla(\eta+B)|^2}} &=
        \frac{\Delta (\eta+B)}{\sqrt{1+|\nabla(\eta+B)|^2}} \\
     &\qquad + \ang{\nabla(\eta+B), \nabla\Big(\frac{1}{\sqrt{1+|\nabla(\eta+B)|^2}}\Big)},\end{split}
    \end{align}
and we want to estimate the terms on the right hand side. First note that
    \begin{equation}\label{laplace}
        \Delta (\eta+B)(r,\vartheta) = - k f(r)^{-2} + f(r)^{-2}\Delta^{\Ss} b(\vartheta)
    \end{equation}
and
    \begin{equation*}
      |\nabla(\eta+B)(r,\vartheta)|^2 = \eta_r(r)^2 + f(r)^{-2} |\nabla^{\Ss}b(\vartheta)|^2.
    \end{equation*} 
Hence the second term on the right hand side of \eqref{mineqetaB} becomes
\begin{align*}
     & \ang{\nabla(\eta+B), \nabla\Big(\frac{1}{\sqrt{1+|\nabla(\eta+B)|^2}}\Big)} =
      \big( 1+\eta_r^2 + f^{-2}|\nabla^{\Ss} b|^2\big)^{-3/2} \\ &\qquad \cdot
      \Big( -\eta_r^2\eta_{rr} + \eta_r f_r |\nabla^{\Ss} b|^2 f^{-3}
      - f^{-4}\bigl\langle \nabla^{\Ss}b,\nabla^{\Ss}\bigl(|\nabla^{\Ss}b|^2\bigr)\bigr\rangle_{\Ss}/2 \Big)\\
&\qquad =
   \big( 1+\eta_r^2 + f^{-2}|\nabla^{\Ss} b|^2\big)^{-3/2}  
  \Big( -\eta_r^2\eta_{rr} + \eta_r f_r |\nabla^{\Ss} b|^2 f^{-3}\\
&\qquad\qquad      - f^{-4}\Hess^{\Ss}b(\nabla^{\Ss}b,\nabla^{\Ss}b)\Big),
           \end{align*}
where $\Hess^{\Ss}$ is the Hessian on $\Ss^{n-1}$. 
Using \eqref{mineqetaB} and \eqref{laplace} we get  
        \begin{align}\label{etaBestim}
        \dv &\frac{\nabla(\eta+B)}{\sqrt{1+|\nabla(\eta+B)|^2}} 
        = \big( 1+\eta_r^2 +
          f^{-2}|\nabla^{\Ss}b|^2 \big)^{-3/2}
          \bigg(  - \frac{k}{f^2} +\frac{\Delta^{\Ss} b}{f^2}-
            \frac{k \eta_r^2}{f^2}  \nonumber \\
           &\qquad \ + \frac{\eta_r^2 \Delta^{\Ss} b}{f^2}  - \frac{k|\nabla^{\Ss}b|^2}{f^4}  
         +  \frac{|\nabla^{\Ss}b|^2 \Delta^{\Ss} b}{f^4} - \eta_r^2\eta_{rr} + \frac{\eta_r f_r |\nabla^{\Ss}b|^2}{f^3} \nonumber
        \\
            &\qquad \  - \frac{\Hess^{\Ss} b\bigl(\nabla^{\Ss}b,\nabla^{\Ss}b\bigr)}{f^4}
            \bigg) \nonumber
         \\
        &= \big( 1+\eta_r^2 + f^{-2}|\nabla^{\Ss}b|^2 \big)^{-3/2}
               \bigg( f^{-2} \big( -k + \Delta^{\Ss}b - k\eta_r^2 + \eta_r^2 \Delta^{\Ss}b \big) \nonumber
          \\
            &\qquad \  
         + f^{-4} \Big( -k |\nabla^{\Ss}b|^2 + |\nabla^{\Ss}b|^2 \Delta^{\Ss}b - \Hess^{\Ss}b\bigl(\nabla^{\Ss}b,\nabla^{\Ss}b\bigr)\Big)  
        \\
            &\qquad \ -\eta_r^2 \Big( k(n-1) \frac{f_r}{f^n} \int_1^r f(s)^{n-3} ds - k f^{-2} \Big)
            + \eta_r f_r |\nabla^{\Ss}b|^2f^{-3}  \bigg)   \nonumber
        \\
        &= \big( 1+\eta_r^2 + f^{-2}|\nabla^{\Ss}b|^2 \big)^{-3/2}
                    \bigg( f^{-2} \big( -k + \Delta^{\Ss}b  + \eta_r^2 \Delta^{\Ss}b \big)  \nonumber
          \\
            &\qquad \   
         + f^{-4} \Big( -k |\nabla^{\Ss}b|^2 + |\nabla^{\Ss}b|^2 \Delta^{\Ss}b - \Hess^{\Ss}b\bigl(\nabla^{\Ss}b,\nabla^{\Ss}b\bigr)\Big)  \nonumber
        \\
            &\qquad \ -\eta_r^2 \Big( k(n-1) \frac{f_r}{f^n} \int_1^r f(s)^{n-3} ds  \Big)
            + \eta_r f_r |\nabla^{\Ss}b|^2f^{-3}  \bigg)   \nonumber
        \\
&\le \big( 1+\eta_r^2 + f^{-2}|\nabla^{\Ss}b|^2 \big)^{-3/2}
                    \bigg( f^{-2} \big( -k + \Delta^{\Ss}b  + \eta_r^2 \Delta^{\Ss}b \big)  \nonumber
          \\
            &\qquad \  
         + f^{-4}|\nabla^{\Ss}b|^2 \big( -k  +  \Delta^{\Ss}b +|\Hess^{\Ss}b| \big)  \nonumber
        \\
            &\qquad \ -\eta_r^2 \Big( k(n-1) \frac{f_r}{f^n} \int_1^r f(s)^{n-3} ds  \Big)
            + \eta_r f_r |\nabla^{\Ss}b|^2f^{-3}  \bigg)   \nonumber
        \\
        &\le 0\nonumber
\end{align} 
when we choose $r$ large enough and then $k \ge ||b||_{C^2}$ large enough.
Note that $\Ss^{n-1}$ is compact so $||b||_{C^2}$ is bounded.
Then the computation above shows that
        \begin{equation*}
        \dv \frac{\nabla(\eta+ B)}{\sqrt{1+|\nabla(\eta+B)|^2}} \le 0
        \end{equation*}
for $r$ and $k$ large enough. In particular, $\eta + B$ is a supersolution to the minimal graph
equation in $M\setminus B(o,r_0)$ for some $r_0$.
       
Choose $k$ so that \eqref{etaBestim} holds and $\eta  > 2 \max |B|$ on the geodesic sphere $\p B(o,r_0)$.
Then  
$a \coloneqq \min_{\p B(o,r_0)} (\eta +  B) >  \max B$.
Since $\eta(r)\to 0$ as $r\to\infty$, the function 
        \[
                w(x) \coloneqq \begin{cases}
                        \min\{(\eta +  B)(x) , a \} &\text{ if } x\in M \setminus B(o,r_0); \\
                        a &\text{ if } x\in B(o,r_0)
                \end{cases}
        \]
is continuous in $\bar{M}$ and coincide with $b$ on $\pinf M$. Moreover, $w$ is a global upper barrier 
for the asymptotic Dirichlet problem with the boundary values $b$ on $\pinf M$.
By replacing $\eta$ with $-\eta$ we obtain the global lower barrier $v$,
\[
v(x) \coloneqq \begin{cases}
                        \max\{(-\eta +  B)(x) , d \} &\text{ if } x\in M \setminus B(o,r_0); \\
                        d &\text{ if } x\in B(o,r_0),
                \end{cases}
\]
where $d=\max_{\p B(o,r_0)} (-\eta +  B)$. Notice that $v\le B\le w$ by construction. 

Next we solve the Dirichlet problem
        \[
                \begin{cases}
                \dv \dfrac{\nabla u_\ell}{\sqrt{1+|\nabla u_\ell|^2}} = 0 &\text{ in } B(o,\ell); \\
                u_\ell|\p B(o,\ell) = B| \p B(o,\ell)
                \end{cases}
        \]
in geodesic balls $B(o,\ell)$, with $\ell\ge r_0$. The existence of barrier functions implies that
        \[
          v \le u_\ell \le w
        \]
on $\p B(o,\ell)$ for all $\ell \ge r_0$.        
Hence, by the maximum principle, $(u_\ell)$ is a bounded sequence and we may apply gradient estimates in compact
subsets of $M$ to find a subsequence,
still denoted by $(u_\ell)$, that converges uniformly on compact subsets in the $C^2$-norm to an entire solution $u$.
The PDE regularity theory implies that $u\in C^\infty(M)$. Moreover,
$v \le u \le w$ and hence it follows that $u$ extends continuously to the boundary $\pinf M$ and has the
boundary values $b$.
    
Suppose then that $\theta\in C(\pinf M)$. Again we interpret $\pinf M$ as $\Ss^{n-1}\subset T_o M$. Let $b_i$ be a sequence of smooth
functions converging uniformly to $\theta$. For each $i$, let $u_i\in C(\bar{M})$ be a solution to \eqref{mingraph} in $M$ with 
$u_i|\pinf M =b_i$. Then the sequence $(u_i)$ is uniformly bounded and consequently their gradients $|\nabla u_i|$ are uniformly bounded. By a diagonal 
argument we find a subsequence that converges locally uniformly with respect to $C^2$-norm to an entire $C^\infty$-smooth solution $u$ of \eqref{mingraph}
that is continuous in $\bar{M}$ with $u|\pinf M=\theta$.          
     
For the uniqueness, assume that $u$ and $\tilde u$ are solutions to the minimal graph equation, continuous
up to the boundary, and $u=\tilde u$ on $\pinf M$. Assume that there exists $y\in M$ with $u(y) > \tilde u(y)$.
Now denote $\delta = (u(y) - \tilde u(y))/2$ and let $U\subset \{x\in M \colon u(x) > \tilde u(x) +\delta\}$
be the component containing the point $y$. Since $u$ and $\tilde u$ are continuous functions that coincides
on the boundary $\pinf M$, it follows that $U$ is relatively compact open subset of $M$. Moreover,
$u = \tilde u +\delta$ on $\p U$, which implies $u = \tilde u+\delta$ in $U$. This is a contradiction
since $y \in U$.
\end{proof}

In terms of the curvature bounds, we obtain the following corollary; see \cite[Theorem 2]{march} or the proof of Corollary~\ref{rotsymplaplacecor}.
\begin{cor}\label{rotsymmingraphcor}
Let $M$ be a rotationally symmetric $n$-dimensional Cartan-Hadamard manifold whose radial sectional curvatures outside a compact set 
	satisfy the upper bounds
	 \begin{equation}\label{min-n=2}
	K(P_x) \le - \frac{1 + \ve }{r(x)^2 \log r(x)}, \quad \text{if } n=2 
	 \end{equation}
      	and
   	 \begin{equation}\label{min-n>2}
  	K(P_x) \le - \frac{1/2 + \ve }{r(x)^2 \log r(x)}, \quad \text{if } n\ge3. 	 
   	 \end{equation}
     		Then the asymptotic Dirichlet problem for the minimal graph equation \eqref{mingraph} is
	solvable with any continuous boundary data on $\pinf M$. 
In particular, there are non-constant bounded entire solutions of \eqref{mingraph} in $M$.
\end{cor}
Indeed, the radial curvature assumptions \eqref{min-n=2} and \eqref{min-n>2} imply the integral condition \eqref{intcond}.
\subsection{$p$-Laplacian} 
Similar approach works also for the $p$-Laplacian and we prove the following existence result for $p\in(2,n)$.
The case $p=2$ equals to the case of usual harmonic functions, which is already known, and the case $p\ge n$ 
is discussed in Section \ref{parabol_sec}. The case $1<p<2$ remains open.
\begin{thm}\label{plaplaceexistence}
Let $p\in(2,n)$ and assume that 
	\begin{equation}\label{rotsymplaplaceint}
	\int_1^{\infty}\Big( f(s)^{\beta} \int_s^{\infty} f(t)^{\alpha} dt \Big) ds < \infty,
	\end{equation}
	$\alpha = -(n-1)/(p-1)$ and $\beta = (n-2p+1)/(p-1)$, i.e. $\alpha+\beta =-2$.  
	Then the asymptotic Dirichlet problem for the $p$-Laplacian is uniquely solvable for any 
	continuous boundary data on $\pinf M$, in particular, 	
	there exist entire non-constant bounded $p$-harmonic functions.
\end{thm} 

\begin{proof}
 Again we interpret $\pinf M$ as $\Ss^{n-1}$. Let $b\colon \Ss^{n-1} \to \R$ be a smooth non-constant 
 function such that $|\Hess^{\Ss}b|<\varepsilon$, where $\varepsilon>0$ will be specified later. 
 Define $B\colon M \setminus 
\{o\} \to \R,$ $B(\exp(r\vartheta))=B(r,\vartheta) = b(\vartheta), \, \vartheta \in \Ss^{n-1}\subset T_oM$.
 Similarly as in the proof of Theorem \ref{rotsymmingraph} we define a function
    \[
     \eta(r) = \int_r^\infty f^\alpha(t) \int_1^t f^\beta(s)ds dt,
    \]
where $\alpha$ and $\beta$ are constants to be determined later. We show that the function $\eta + B,\ \eta(x)\coloneqq \eta\bigl(r(x)\bigr)$,  
is a supersolution for the $p$-Laplace equation, i.e.
\[
\Delta_p(\eta+B)\coloneqq \dv\bigl(|\nabla(\eta+B)|^{p-2}\nabla(\eta+B)\bigr)\le 0.
\] 
Since $\eta_r<0$ and $B_r=0$, we have $|\nabla(\eta+B)|>0$ in $M\setminus\{o\}$. 
First we compute
\begin{align*}
 \Delta_p (\eta +B) &
 = \dv (|\nabla (\eta +B)|^{p-2} \nabla (\eta +B))\\
 &= |\nabla (\eta +B)|^{p-2}  \Delta (\eta +B) \\
 &\qquad + \frac{p-2}{2}|\nabla(\eta+B)|^{p-4} \ang{\nabla(\eta +B), \nabla \bigl(|\nabla (\eta +B)|^2\bigr)}\\
 &= |\nabla (\eta +B)|^{p-4}
 \bigg[\bigg(\eta_r^2 + \frac{|\nabla^{\Ss}b|^2}{f^2}\bigg)\bigg(\eta_{rr}+(n-1)\frac{f_r\eta_r}{f} +\frac{\Delta^{\Ss}b}{f^2}\bigg)\\
&\qquad +(p-2)\bigg(\eta_r^2\eta_{rr}-\frac{\eta_r f_r |\nabla^{\Ss}b|^2}{f^3} + 
\frac{\Hess^{\Ss}b\bigl(\nabla^{\Ss}b,\nabla^{\Ss}b\bigr)}{f^4}\bigg)\bigg].
\end{align*}
Since we are interested in the sign of $\Delta_p(\eta+B)$, we may just consider the term inside the brackets. 
Again, by straightforward computation, we obtain
  \begin{align*}
& \frac{\Delta_p (\eta +B)}{|\nabla (\eta +B)|^{p-4}} = 
\left(\eta_r^2 + \frac{|\nabla^{\Ss}b|^2}{f^2}\right)\left(\eta_{rr}+(n-1)\frac{f_r\eta_r}{f} +\frac{\Delta^{\Ss}b}{f^2}\right)\\
&\qquad +(p-2)\left(\eta_r^2\eta_{rr}-\frac{\eta_r f_r |\nabla^{\Ss}b|^2}{f^3} +
\frac{\Hess^{\Ss}b\bigl(\nabla^{\Ss}b,\nabla^{\Ss}b\bigr)}{f^4} \right)\\   
&\quad =\eta_r^2\left((p-1)\eta_{rr} + (n-1)\frac{f_r\eta_r}{f} +  \frac{\Delta^{\Ss}b}{f^2}\right) \\
& \qquad + \frac{|\nabla^{\Ss}b|^2}{f^2}\left( \eta_{rr} + (n-p+1)\frac{f_r\eta_r}{f} +  \frac{\Delta^{\Ss}b}{f^2}\right)
+ (p-2)f^{-4}\Hess^{\Ss}b\bigl(\nabla^{\Ss}b,\nabla^{\Ss}b\bigr)\\
&\quad =\eta_r^2\left(-\bigl((p-1)\alpha+n-1\bigr)f^{\alpha-1}f_r\int_1^r f^{\beta}(s)ds - (p-1)f^{\alpha+\beta}+\frac{\Delta^{\Ss}b}{f^2}\right)\\
&\qquad + \frac{|\nabla^{\Ss}b|^2}{f^2}\left((\alpha+n-p+1)\frac{f_r\eta_r}{f}
   - f^{\alpha+\beta} + \frac{\Delta^{\Ss}b}{f^2}\right)\\
&\qquad +(p-2)f^{-4}\Hess^{\Ss}b\bigl(\nabla^{\Ss}b,\nabla^{\Ss}b\bigr).   
 \end{align*}
Then choosing $\alpha = -(n-1)/(p-1)$ and $\beta = (n-2p+1)/(p-1)$, i.e. such that $\alpha + \beta=-2$,
and recalling that $p\in(2,n)$ and $\eta_r<0$ we see that
\begin{align*}
 \frac{\Delta_p (\eta +B)}{|\nabla (\eta +B)|^{p-4}} & = 
 \frac{\eta_r^2}{f^2}(\Delta^{\Ss}b-p+1) + \frac{|\nabla^{\Ss}b|^2}{f^4}\left( \frac{(n-p)(p-2)ff_r\eta_r}{p-1}- 1 +\Delta^{\Ss}b\right)\\
&\quad +(p-2)f^{-4}\Hess^{\Ss}b\bigl(\nabla^{\Ss}b,\nabla^{\Ss}b\bigr) \\  
&\le \frac{\eta_r^2}{f^2}\big(-p+1+\Delta^{\Ss}b\big)\\
&\quad + \frac{|\nabla^{\Ss}b|^2}{f^4}\left( \frac{(n-p)(p-2)ff_r\eta_r}{p-1}- 1 +\Delta^{\Ss}b
+(p-2)|\Hess^{\Ss}b|\right)\\
& \le 0
\end{align*}
when $|\Hess^{\Ss}b|<\varepsilon$, with $\varepsilon>0$ small enough, e.g.     
$\varepsilon<\min\bigl(\tfrac{p-1}{n-1},\tfrac{1}{n+p-3}\bigr)$.
Hence $\eta + B$ is a $p$-supersolution in $M\setminus\{o\}$. Similarly, we obtain an estimate 
\begin{align*}
 \frac{\Delta_p (-\eta +B)}{|\nabla (-\eta +B)|^{p-4}} & = 
 \frac{\eta_r^2}{f^2}(p-1+\Delta^{\Ss}b) + \frac{|\nabla^{\Ss}b|^2}{f^4}\left( \frac{-(n-p)(p-2)ff_r\eta_r}{p-1} + 1 +\Delta^{\Ss}b\right)\\
&\quad +(p-2)f^{-4}\Hess^{\Ss}b\bigl(\nabla^{\Ss}b,\nabla^{\Ss}b\bigr) \\  
&\ge \frac{\eta_r^2}{f^2}\big(p-1+\Delta^{\Ss}b\big)\\
&\quad + \frac{|\nabla^{\Ss}b|^2}{f^4}\left( \frac{-(n-p)(p-2)ff_r\eta_r}{p-1}+ 1 +\Delta^{\Ss}b
-(p-2)|\Hess^{\Ss}b|\right)\\
& \ge 0,
\end{align*}
and therefore $-\eta+B$ is a $p$-subsolution in $M\setminus\{o\}$. Notice that $k(\eta+B)$ is a $p$-supersolution in $M\setminus\{o\}$ for 
all $k\ge 0$ and similarly $k(-\eta+B)$ is a $p$-subsolution in $M\setminus\{o\}$. Hence the assumption $|\Hess^{\Ss}b|<\varepsilon$ is not a restriction.
The asymptotic Dirichlet problem with any continuous boundary data $\varphi\in C(\pinf M)$ can then be uniquely solved either by Perron's method with a 
suitable choice of the function $b$ or approximating the given $\varphi\in C(\pinf M)$ by functions $b_i\in C^{\infty}$. 
For reader's convenience we sketch the latter argument. Indeed, for each $\ell\in\N$ we first 
 solve the Dirichlet problem
        \[
                \begin{cases}
                \dv \big(|\nabla u_\ell|^{p-2}\nabla u_\ell\big) = 0 &\text{ in } B(o,\ell); \\
                u_\ell|\p B(o,\ell) = B| \p B(o,\ell)
                \end{cases}
        \]
in geodesic balls $B(o,\ell)$. Then $-\eta+B\le u_\ell\le\eta+B$
and there exists a subsequence of $(u_\ell)$ converging uniformly on compact sets to an entire $p$-harmonic function $u$. Moreover, $u$ extends continuously to $\pinf M$ and has the boundary values $b$. Finally, given $\varphi\in C(\pinf M)$ we again interpret $\pinf M$ as $\Ss^{n-1}\subset T_o M$ and choose a sequence $b_i\in\C^\infty (\Ss^{n-1})$ converging uniformly to $\varphi$. For each $i$, let $u_i\in C(\bar{M})$ be a solution to \eqref{plaplaceeqn} in $M$ with 
$u_i|\pinf M =b_i$. Then the sequence $(u_i)$ has a subsequence that converges locally uniformly to an entire $p$-harmonic function $u$ that is continuous in $\bar{M}$ with $u|\pinf M=\varphi$. Hence $u$ solves the asymptotic Dirichlet problem 
with boundary data $\varphi\in C(\pinf M)$. The uniqueness of the solution follows exactly as in the case of the minimal graph equation. 
See \cite{Va1} and \cite{Ho} for further details.
\end{proof}

In terms of curvature bounds, we obtain the following corollary.
    \begin{cor}\label{rotsymplaplacecor}
Let $M$ be a rotationally symmetric $n$-dimensional Cartan-Hadamard manifold, with $n\ge 3$, whose radial sectional curvatures outside a compact set 
	satisfy 
  \begin{equation}\label{rotsymcurvas}
      K(P_x) \le - \frac{1/2 + \ve }{r(x)^2 \log r(x)}.
      \end{equation}
 Then the asymptotic Dirichlet problem for the $p$-Laplacian, with  $p\in(2,n)$, is uniquely solvable for any continuous boundary data on $\pinf M$. 
 In particular, there exist non-constant bounded $p$-harmonic functions on $M$.      
    \end{cor}
\begin{proof}
It is enough to show that the curvature assumption \eqref{rotsymcurvas} implies finiteness of the 
integral 
  \begin{equation*}
	\int_1^{\infty}\Big( f(s)^{\beta} \int_s^{\infty} f(t)^{\alpha} dt \Big) ds < \infty,
	\end{equation*}
 where $\alpha = -(n-1)/(p-1)$ and $\beta = (n-2p+1)/(p-1)$, i.e. $\alpha+\beta =-2$. Although this
 seems more complicated than the situation with \eqref{intcond}, it is essentially the same because
 $\alpha$ and $\beta$ are chosen so that $\alpha+\beta=-2$ which is same as the sum of the exponents
 in \eqref{intcond}. For the sake of convenience, we give some details.
 
 As in \cite{march}, define $\phi(r) = r(\log r)^c$, $c>0$. Choose $a>1$ such that $\phi'(a)>0$,
 $\phi''(a)>0$ and let $g(r) = (\phi(r+a) - \phi(a))/\phi'(a)$. Then $g(0)=0,\ g^\prime(0)=1$, and 
$-g''(r)/g(r) \le 0$ behaves asymptotically as 
    \[
    -\frac{g''}{g}(r)\approx
     -\frac{\phi''}{\phi}(r) = - \frac{c}{r^2 \log r}\left( 1 + \frac{c-1}{\log r} \right)
    \]
    as $r\to\infty$.
Applying \cite[Lemma 5]{march}, we see that \eqref{rotsymplaplaceint} is equivalent to the finiteness
of the similar integral condition for $g$. Moreover, $g(r)$ behaves asymptotically as $\phi(r)$, so
it is enough to show 
    \begin{equation}\label{intcondforphi}
     \int_2^{\infty}\Big( \phi(s)^{\beta} \int_s^{\infty} \phi(t)^{\alpha} dt \Big) ds < \infty.
    \end{equation}
But $\int_s^\infty \phi(t)^\alpha dt$ behaves asymptotically as 
\[    
  \frac{s^{\alpha+1}(\log s)^{c\alpha}}{-\alpha-1},  
 \]
 and therefore
 \[
\phi(s)^{\beta} \int_s^{\infty} \phi(t)^{\alpha}dt\approx \frac{p-1}{n-p}\frac{1}{s(\log s)^{2c}}
\]
as $s\to\infty$. 
Hence \eqref{intcondforphi} holds if and only if $c> 1/2$.
\end{proof}

\section{$p$-parabolicity when $p\ge n$}\label{parabol_sec}

In this section we show that the upper bound $p < n$ in Theorem \ref{plaplaceexistence} cannot be
improved. Namely there exist manifolds that satisfy the curvature assumption \eqref{rotsymplaplaceint}
and are $p$-parabolic when $p\ge n$. Recall that a Riemannian manifold $N$ is 
called \emph{$p$-parabolic}, $1<p<\infty$, if
    \[
     \pcap(K,N) = 0
    \]
for every compact set $K\subset N$. Here the \emph{$p$-capacity} of the pair $(K,N)$ is defined as
\[
\pcap(K,N)=\inf_{u}\int_{N}|\nabla u|^p d\mu_0,
\]
where the infimum is taken over all $u\in C^{\infty}_0(N)$, with $u|K\ge 1$.
In \cite[Proposition 1.7]{holDuke} it was shown that a complete Riemannian manifold is $p$-parabolic
if
  \[
   \int^\infty \left( \frac{t}{V(t)} \right)^{1/(p-1)} dt = \infty,
  \]
where $V(t) = \mu_0(B(o,t))$ and $o\in N$ is a fixed point. We apply this to get the following result.
\begin{thm}\label{parabthm}
 Let $\alpha > 0$ be a constant and assume that $M$ is a complete $n$-dimensional Riemannian manifold whose radial sectional curvatures 
 satisfy 
  \begin{equation}\label{paracurv}
    K_M(P_x) \ge - \frac{\alpha}{r(x)^2 \log r(x)}
  \end{equation}
for every $x$ outside some compact set and every $2$-dimensional subspace $P_x\subset T_xM$ containing $\nabla r(x)$. Then
$M$ is $p$-parabolic 
\begin{enumerate}
\item[(a)] if $p=n$ and $0<\alpha\le 1$; or
\item[(b)] $p>n$ and $\alpha>0$.
\end{enumerate}
\end{thm}

\begin{proof}
 Let $R>1$ be so large that the curvature assumption \eqref{paracurv} holds in $M\setminus B(o,R)$
 and denote
  \[
    B = \inf \big\{K_M(P_x) \colon x \in \bar B(o,R-1)\big\} > -\infty.
  \]
Let $k\colon[o,\infty) \to (-\infty,0]$ be a smooth function that is constant in some neighborhood of $0$,
$k(t)\le B$ for $t\in [0,R-1]$, $k(t) \le -\alpha/(t^2\log t)$ for $t\in [R-1,R]$ and $k(t)= -\alpha/(t^2\log t)$
for all $t\ge R$. Then the sectional curvatures of $M$ are bounded from below by $k\circ r$.
Applying the Bishop-Gromov volume comparison theorem 
we obtain
  \[
V(r)= \mu_0 \big(B(o,r)\big) \le C r^n (\log r)^{\alpha(n-1)}
  \]
for some constant $C$ and for $r\ge R$ large enough.
  
Consider first the case $p=n$. Then 
  \begin{align*}
   \int_R^\infty \left( \frac{t}{V(t)} \right)^{1/(n-1)} dt 
   &\ge c\int_R^\infty \left( \frac{t}{t^n (\log t)^{\alpha(n-1)}} \right)^{1/(n-1)} dt \\
   &= c\int_R^\infty  \frac{1}{t (\log t)^{\alpha}}  dt = \infty
  \end{align*}
if $0<\alpha \le 1$. This proves the first case. On the other hand, if $p>n$, we have
$t^{n-1}(\log t)^{\alpha(n-1)} \le t^{n-1}(\log t)^{\alpha(p-1)}$ and
    \[
     \frac{t^{(n-1)/(p-1)}(\log t)^{\alpha}}{t} \longrightarrow 0
    \]
for any $\alpha > 0$ as $t\to\infty$, and hence
    \begin{align*}
     \int_R^\infty \left( \frac{t}{V(t)} \right)^{1/(p-1)} dt &\ge
     c\int_R^\infty \left( \frac{1}{t^{n-1}(\log t)^{\alpha(p-1)}} \right)^{1/(p-1)} dt \\
     &= c\int^\infty \frac{1}{t^{(n-1)/(p-1)}(\log t)^{\alpha}}  dt = \infty
    \end{align*}
for any $\alpha >0$. This proves the second case.  
  \end{proof}

%\bibliography{referencesADP2}
%\bibliographystyle{plain}

\end{document}